\documentclass[11pt,a4paper]{amsart}

\usepackage{amsmath,amssymb,amsthm,mathtools,mathrsfs}
\usepackage[margin=1.1in]{geometry}
\usepackage{microtype}
\usepackage{tikz}
\usetikzlibrary{arrows.meta,positioning}
\usepackage[colorlinks=true,linkcolor=blue,citecolor=blue,urlcolor=blue]{hyperref}

\newtheorem{theorem}{Theorem}[section]
\newtheorem{lemma}[theorem]{Lemma}
\newtheorem{proposition}[theorem]{Proposition}
\newtheorem{corollary}[theorem]{Corollary}
\theoremstyle{definition}
\newtheorem{definition}[theorem]{Definition}
\theoremstyle{remark}
\newtheorem{remark}[theorem]{Remark}

\numberwithin{equation}{section}

\newcommand{\Z}{\mathbb Z}
\newcommand{\R}{\mathbb R}
\newcommand{\Pp}{\mathbb P}
\newcommand{\Ee}{\mathbb E}
\newcommand{\one}{\mathbf 1}
\newcommand{\wtS}{\widetilde S_n}
\newcommand{\cD}{\mathcal D}
\newcommand{\cS}{\mathcal S}
\newcommand{\cH}{\mathcal H}
\newcommand{\cF}{\mathcal F}
\newcommand{\Law}{\operatorname{Law}}
\newcommand{\TV}{\mathrm{TV}}
\newcommand{\id}{\mathrm{id}}
\newcommand{\floor}[1]{\left\lfloor #1\right\rfloor}
\newcommand{\ceil}[1]{\left\lceil #1\right\rceil}
\newcommand{\abs}[1]{\left|#1\right|}
\newcommand{\norm}[1]{\left\lVert #1\right\rVert}
\newcommand{\pospart}[1]{\left(#1\right)_{+}}
\newcommand{\modn}[1]{\overline{#1}}

\title[Periodic ASEP and affine Hecke walks]
{Periodic ASEP and random walks on affine Hecke algebra}

\author{Alexey Bufetov}
\address{Institute of Mathematics, Leipzig University, Germany}
\email{alexey.bufetov@gmail.com}

\author{Nimisha Pahuja}
\address{ICTS-TIFR, Bengaluru, India}
\email{nimisha.pahuja@icts.res.in}

\date{}

\begin{document}
	
	\maketitle
	
	\begin{abstract}
		We study the large time behavior of a periodic ASEP with fixed period started from the step initial condition. We establish the limit shape theorem and the asymptotic distribution of a single second class particle. Our single-species results extend the previous results of Lam and Ayyer-Linusson about periodic TASEP to the ASEP case. These results can also be interpreted as describing the limit behavior of a random walk on affine Hecke algebra. 
	\end{abstract}
	
	\section{Introduction}
	\label{sec:Intro}
	
	The asymmetric simple exclusion process (ASEP) is one of the most extensively
	studied interacting particle systems, owing both to its natural role as a
	model of nonequilibrium statistical mechanics and to its deep connections
	with several areas of mathematics.  Among these connections is a rich
	algebraic structure; the aspect most relevant to the present paper is the relation
	between ASEP and Hecke algebras.
	
	The connection between asymmetric exclusion processes and Hecke algebras
	goes back to algebraic formulations of one-dimensional reaction--diffusion
	systems; see, in particular, \cite{AlcarazDrozHenkelRittenberg}.  In
	\cite{BufetovHecke}, it was shown that a large class of multispecies
	interacting particle systems, including ASEP, can be viewed as random walks
	on Hecke algebras.  The aim of the present paper is to study the affine
	type~$A$ version of this construction and its long-time behavior in greater
	detail.  As we will see, this setup is naturally related to ASEP on a ring.
	
	Fix $n\ge2$ and $q\in[0,1)$.  Our basic process is the continuous-time
	random walk generated by stochastic multiplication in the affine Hecke
	algebra $\mathscr H_q(\wtS)$.  A clock associated with a simple reflection
	$s_r$ acts simultaneously on all pairs of positions congruent to $r,r+1$
	modulo $n$: a length-increasing exchange occurs with rate $1$, whereas the
	reverse exchange occurs with rate $q$.  Its one- and multispecies
	projections are the periodic ASEPs described in more detail below. 
	
	For $q=0$ the study of this process was
	pioneered in \cite{LamAffineShape}, where a law of large numbers for
	the affine $0$-Hecke walk was proved and an explicit limiting shape was conjectured. This
	conjecture, together with several related ones, was subsequently proved in
	\cite{AyyerLinussonLam}.  See also \cite{AasAyyerLinussonPotka} for TASEP in
	other classical affine types.
	
	The present paper extends these limit-shape results from the totally
	asymmetric case $q=0$ to ASEP with $0<q<1$.  Our main results are
	Theorem~\ref{thm:limit-shape-step}, Theorem~\ref{thm:second-class-step}, and Theorem~\ref{thm:shifted-second-class}. Let us formulate them.
	
	\begin{figure}[ht]
		\centering
		\includegraphics[width=0.7\textwidth]{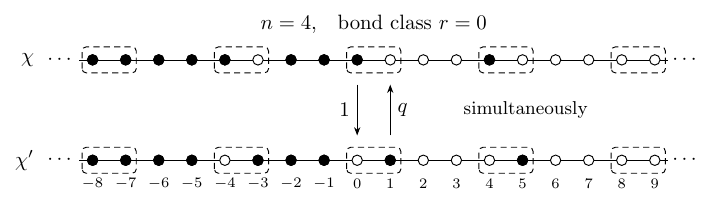}
		\caption{Update of periodic ASEP.}
		\label{fig:picture}
	\end{figure}
	
	The single-species periodic ASEP is a process on configurations
	$\chi:\Z\to\{0,1\}$ that are nonincreasing along every arithmetic
	progression modulo $n$.  For each $r\in\Z/n\Z$, an independent rate-one
	clock simultaneously updates all bonds
	\[
	(r+kn,r+kn+1),\qquad k\in\Z.
	\]
	At a ringing time, pairs of the form $10$ are exchanged to $01$, while
	pairs of the form $01$ are exchanged to $10$ with probability $q$; equal
	pairs remain unchanged.  The admissibility condition ensures that all
	unequal pairs updated by the same clock have the same orientation.  We
	start from the step initial condition
	\begin{equation}\label{eq:introduction-step-initial-condition}
		\chi_0(x)=\one_{\{x\le0\}},\qquad x\in\Z,
	\end{equation}
	and denote the resulting process by $(\chi_t)_{t\ge0}$.  For $x\in\Z$, let
	\begin{equation}\label{eq:introduction-particle-count}
		N_t(x)=\sum_{y>x}\chi_t(y)
	\end{equation}
	be the number of particles strictly to the right of $x$.  Define the
	characteristic velocities
	\begin{equation}\label{eq:introduction-velocities}
		v_k=\frac{(1-q)(n+1-2k)}{n-1},
		\qquad k\in[n],
	\end{equation}
	and the piecewise-linear function
	\begin{equation}\label{eq:introduction-limit-shape}
		\cH(u)=\frac1n\sum_{k=1}^n\pospart{v_k-u},
		\qquad u\in\R, \qquad \pospart{x}:=\max(x,0).
	\end{equation}
	Theorem~\ref{thm:limit-shape-step}(ii) states that, almost surely,
	\begin{equation}\label{eq:introduction-uniform-height-limit}
		\sup_{u\in\R}
		\abs{\frac{N_t(\floor{ut})}{t}-\cH(u)}
		\longrightarrow0.
	\end{equation}
	
	We next formulate our second-class-particle theorems.  Use the ordered
	alphabet $\{1<2<+\infty\}$, where $1$ denotes a first class particle, $2$ a
	second class particle, and $+\infty$ a hole.  The same periodic clocks are
	used, with an adjacent pair $ab$ exchanged to $ba$ at rate $1$ when $a<b$
	and at rate $q$ when $a>b$.  For the step initial condition with one second
	class particle,
	\begin{equation}\label{eq:introduction-second-class-step}
		\eta_0(x)=
		\begin{cases}
			1,&x\le-1,\\
			2,&x=0,\\
			+\infty,&x\ge1,
		\end{cases}
	\end{equation}
	let $Q(t)$ denote the position of the second class particle. In Theorem~\ref{thm:second-class-step} we establish that
	\begin{equation}\label{eq:introduction-uniform-second-class-law}
		\frac{Q(t)}{t}\Longrightarrow V,
		\qquad
		\Pp(V=v_k)=\frac1n,
		\qquad k\in[n].
	\end{equation}
	This is a periodic analogue of the random selection of a characteristic by
	a second class particle in a rarefaction fan; cf.
	\cite{FerrariKipnis,FerrariGoncalvesMartin,AmirAngelValko,AggarwalCorwinGhosal}.
	
	A different law appears when the second class particle in the initial configuration is shifted by one
	position to the right, leaving a hole at the origin.  Namely, consider
	\begin{equation}\label{eq:introduction-shifted-second-class-step}
		\widetilde\eta_0(x)=
		\begin{cases}
			1,&x\le-1,\\
			+\infty,&x=0,\\
			2,&x=1,\\
			+\infty,&x\ge2,
		\end{cases}
	\end{equation}
	and let $\widetilde Q(t)$ be the position of its second class particle.
	Theorem~\ref{thm:shifted-second-class} establishes
	\begin{equation}\label{eq:introduction-shifted-speed-law}
		\frac{\widetilde Q(t)}{t}\Longrightarrow\widetilde V,
		\qquad
		\Pp(\widetilde V=v_k)
		=\frac qn+\frac{2(1-q)(n-k)}{n(n-1)},
		\qquad k\in[n].
	\end{equation}
	This is a periodic upgrade of \cite[Theorem 5.2 and Example 5.7]{BorodinBufetovColorPosition}.

	The proof of Theorem~\ref{thm:limit-shape-step} combines the stationary
	current identity for the multispecies ASEP on a ring, obtained from the
	adjacent-correlation formulas derived in \cite{PahujaCorrelations} via the construction of \cite{MartinASEP}, and various new ``soft" probabilistic considerations, which appear in $q>0$ case. The crucial technical statement is Proposition \ref{prop:eventual-separation}. 
	The second-class-particle limit behavior can be studied via the color-position
	symmetry \cite{BorodinBufetovColorPosition,BufetovHecke}, supplied by the affine Hecke algebra. 
	
	\subsection*{Acknowledgements}
	A.~Bufetov was partially supported by the European Research Council (ERC), Grant Agreement No. 101041499. N. Pahuja is grateful to MPI Leipzig for the hospitality during the visit to Leipzig, during which part of this work was done. A humanly generated complete draft of the text was lightly edited with the help of ChatGPT. 
	
	\section{Preliminaries}\label{sec:preliminaries}
	
	Throughout the paper, we fix an integer $n\ge 2$ and an asymmetry parameter
	$q\in[0,1)$.  We write $[n]=\{1,\ldots,n\}$, and all residue classes are
	understood modulo $n$.
	
	\subsection{Periodic ASEP}\label{subsec:periodic-asep}
	
	A one-species configuration is a function
	\[
	\chi:\Z\longrightarrow\{0,1\},
	\]
	where $\chi(x)=1$ means that $x$ is occupied by a particle and $\chi(x)=0$ means that $x$
	is empty.  We call $\chi$ \emph{$n$-admissible} if, for every $r\in\Z/n\Z$,
	the sequence
	\[
	k\longmapsto \chi(r+kn),\qquad k\in\Z,
	\]
	is nonincreasing.
	
	The periodic ASEP can be defined via a marked graphical
	construction.  For every $r\in\Z/n\Z$, let $\mathcal N_r$ be an independent
	rate-one Poisson process on $\mathbb{R}_{\ge 0}$, and additionally attach to every point of $\mathcal N_r$ an
	independent mark, uniformly distributed on $[0,1]$.  Almost surely, no two of
	the finitely many clocks ring at the same time.  At a point of
	$\mathcal N_r$, inspect simultaneously all pairs
	\[
	\bigl(\chi(r+kn),\chi(r+kn+1)\bigr),\qquad k\in\Z.
	\]
	In this sequence, all pairs with non-coinciding numbers are either all $(1,0)$ or all $(0,1)$. If all such pairs are $(1,0)$, then they are changed to $(0,1)$ with probability 1. If all such pairs are $(0,1)$, then all such pairs are changed to $(1,0)$ when the attached
	mark is at most $q$.  Equal pairs are always left unchanged. Thus a simultaneous right jump has rate $1$,
	while the reverse simultaneous jump has rate $q$.
	
	The step initial condition is
	\begin{equation}\label{eq:step-initial-condition}
		\chi_0(x)=\one_{\{x\le 0\}},\qquad x\in\Z.
	\end{equation}
	
	\begin{remark}\label{rem:ring-special-case}
		Note that one clock updates every edge in a fixed
		congruence class.  If one considers configurations that are constant on each
		arithmetic progression modulo $n$, the process reduces to the usual
		one-species ASEP on a ring with $n$ sites.
	\end{remark}
	
	\subsection{Affine permutations and the affine Hecke algebra}
	\label{subsec:affine-hecke}
	
	\begin{definition}\label{def:affine-permutation}
		An \emph{affine permutation of size $n$} is a bijection $w:\Z\to\Z$ such
		that
		\begin{align}
			w(i+n)&=w(i)+n, && \mbox{for all $i\in\Z$},\label{eq:affine-periodicity}\\
			\sum_{i=1}^n w(i)&=\binom{n+1}{2}.\label{eq:affine-normalization}
		\end{align}
		The group of affine permutations is denoted by $\wtS$.
	\end{definition}
	
	The values $w(1),\ldots,w(n)$ determine $w$ uniquely; we use the window
	notation
	\[
	w=[w(1),\ldots,w(n)].
	\]
	For $r\in\Z/n\Z$, the simple reflection $s_r$ is the affine permutation that
	simultaneously interchanges $r+kn$ and $r+kn+1$ for every $k\in\Z$.  Thus,
	for any integer representative of $r$,
	\[
	s_r(x)=
	\begin{cases}
		x+1,&x\equiv r\pmod n,\\
		x-1,&x\equiv r+1\pmod n,\\
		x,&\text{otherwise}.
	\end{cases}
	\]
	These reflections generate $\wtS$.  We denote the corresponding Coxeter length by $\ell$.
	We shall use the window formula for the length
	\begin{equation}\label{eq:affine-length-window}
		\ell(w)=
		\sum_{1\le i<j\le n}
		\abs{\floor{\frac{w(j)-w(i)}{n}}},
	\end{equation}
	and the right-descent criterion
	\begin{equation}\label{eq:right-descent-criterion}
		\ell(ws_r)=\ell(w)+1
		\quad\Longleftrightarrow\quad
		w(r)<w(r+1),
		\qquad r\in\Z/n\Z.
	\end{equation}
	The facts above are standard; see, for example,
	\cite{BjornerBrenti}.
	
	\begin{definition}\label{def:affine-hecke-algebra}
		The \emph{affine Hecke algebra} $\mathscr H_q(\wtS)$ is the real
		vector space with basis $\{T_w:w\in\wtS\}$ and multiplication determined by
		\begin{equation}\label{eq:hecke-multiplication}
			T_wT_{s_r}=
			\begin{cases}
				T_{ws_r},&\ell(ws_r)=\ell(w)+1,\\[1mm]
				(1-q)T_w+qT_{ws_r},&\ell(ws_r)=\ell(w)-1,
			\end{cases}
			\qquad w\in\wtS,\quad r\in\Z/n\Z.
		\end{equation}
	\end{definition}
	
	The linear map $\iota: \mathscr H_q(\wtS) \to \mathscr H_q(\wtS)$ defined via the images of the basis $\iota(T_w) = T_{w^{-1}}$, for any $w \in \wtS$, is well known to be an anti-involution:
	\begin{equation}\label{eq:hecke-anti-involution}
		\iota(h_1h_2)=\iota(h_2)\iota(h_1),
		\qquad h_1,h_2\in\mathscr H_q(\wtS).
	\end{equation}
	
	\subsection{The affine Hecke walk and its projections}
	\label{subsec:affine-hecke-walk}
	
	The \emph{affine Hecke walk} $(w_t)_{t\ge0}$ is the continuous-time Markov
	chain on $\wtS$ obtained from the same Poisson processes and marks as above.  At a point of
	$\mathcal N_r$, the state $w$ is changed to $ws_r$ if
	$\ell(ws_r)=\ell(w)+1$; if $\ell(ws_r)=\ell(w)-1$, the change is made exactly
	when the attached mark is at most $q$.  Its generator is therefore
	\begin{equation}\label{eq:affine-walk-generator}
		(\mathcal L f)(w)=\sum_{r\in\Z/n\Z}a_r(w)
		\bigl(f(ws_r)-f(w)\bigr),
		\qquad
		a_r(w)=
		\begin{cases}
			1,&\ell(ws_r)=\ell(w)+1,\\
			q,&\ell(ws_r)=\ell(w)-1.
		\end{cases}
	\end{equation}
	Thus the walk is encoded by right multiplication in
	\eqref{eq:hecke-multiplication}; cf. \cite{BufetovHecke}. 
	
	Our convention throughout the paper is that
	\begin{equation}\label{eq:position-to-label-convention}
		w(x)=\text{the label of the particle occupying position }x.
	\end{equation}
	Consequently, the particle with label $i$ is at position $w^{-1}(i)$.  Right
	multiplication by $s_r$ interchanges the labels at every pair of positions
	$r+kn,r+kn+1$, $k\in\Z$.  By \eqref{eq:right-descent-criterion}, if those
	labels are $a$ and $b$, respectively, then during the affine walk they are interchanged at rate $1$
	when $a<b$ and at rate $q$ when $a>b$.  The comparison is the same for all
	$k\in\Z$, because
	\[
	w(r+kn)=w(r)+kn,
	\qquad w(r+kn+1)=w(r+1)+kn.
	\]
	
	Let $\mathcal A$ be a finite totally ordered set and let
	$\varphi:\Z\to\mathcal A$ be nondecreasing.  Define
	\begin{equation}\label{eq:projection-general}
		\eta_w^{\varphi}(x)=\varphi\bigl(w(x)\bigr),
		\qquad x\in\Z.
	\end{equation}
	It is easy to see that this is a Markovian projection of the affine Hecke walk.  
	
	For the one-species projection, take the ordered alphabet
	$\{\mathsf p<\mathsf h\}$ and map labels $j\le0$ to a particle $\mathsf p$
	and labels $j\ge1$ to a hole $\mathsf h$.  After the occupation
	recoding $\mathsf p\mapsto1$, $\mathsf h\mapsto0$, this becomes
	\begin{equation}\label{eq:one-species-projection}
		\chi_w(x)=\one_{\{w(x)\le0\}}.
	\end{equation}
	Thus the order on species is preserved by $\varphi$, whereas the numerical
	occupation coding $1,0$ reverses that order.  If $w_0=\id$, the projection
	is the step configuration \eqref{eq:step-initial-condition}.
	
	Using ordered alphabets with more letters, we obtain via \eqref{eq:projection-general} finite-species periodic ASEP's. In this paper, we will need only two-species process with one second class particle, defined also in Section \ref{sec:Intro}. 
	
	Finally, the anti-involution $\iota$ immediately implies the following
	color-position symmetry.  This is the affine-periodic version of the
	symmetry for colored ASEP in \cite{BorodinBufetovColorPosition}, \cite{BufetovHecke}.
	
	\begin{lemma}[Color-position symmetry]\label{lem:color-position-symmetry}
		If the affine Hecke walk starts from the identity, then, for every $t\ge0$,
		\begin{equation}\label{eq:fixed-time-inversion-symmetry}
			w_t\ \stackrel{d}{=}\ w_t^{-1}.
		\end{equation}
	\end{lemma}
	
	\subsection{Multispecies ASEP on a ring}\label{subsec:ring-asep}
	
	The state space of the ring process is $S_n$, viewed as the set of words
	$\omega=(\omega(1),\ldots,\omega(n))$ containing each label in $[n]$ once.
	At the clock of the oriented edge $(r,r+1)$, with $n+1$ identified with $1$,
	the adjacent entries are interchanged at rate
	\begin{equation}\label{eq:ring-rates}
		ab\longrightarrow ba
		\quad\text{at rate}\quad
		\begin{cases}
			1,&a<b,\\
			q,&a>b,
		\end{cases}
		\qquad a,b\in[n],\quad a\ne b.
	\end{equation}
	For every $q\in[0,1)$ the chain is irreducible on $S_n$. Therefore, it has a unique stationary law,
	denoted by $\pi_n^q$, which is rotation invariant. 
	
	For distinct $i,j\in[n]$, set
	\begin{equation}\label{eq:ring-correlations-definition}
		c_{i,j}^q(n)=\pi_n^q\bigl(\omega(1)=i,\ \omega(2)=j\bigr).
	\end{equation}
	We will need the following result
	
	\begin{theorem}\label{thm:ring-current-identity}
		For $1\le i<j\le n$,
		\begin{equation}\label{eq:ring-current-identity}
			c_{i,j}^q(n)-q\,c_{j,i}^q(n)
			=\frac{(1-q)(j-i)}{n\binom n2}.
		\end{equation}
	\end{theorem}
	
	\begin{proof}
		Let $\widehat c_{a,b}^q(n)$ denote the correlations in the convention of
		\cite{PahujaCorrelations}, in which $ab\to ba$ has rate $1$ for $a>b$ and
		rate $q$ for $a<b$.  Reversing all labels, $a\mapsto n+1-a$, conjugates that
		process to \eqref{eq:ring-rates}; hence
		\begin{equation}\label{eq:correlation-label-reversal}
			c_{i,j}^q(n)=
			\widehat c_{n+1-i,n+1-j}^q(n).
		\end{equation}
		We record the algebraic consequence of
		\cite[Theorem~2.2]{PahujaCorrelations} that is needed below.  Fix $a>b$, put
		$d=a-b$ and $D=n^2(n-1)$, and write
		$[m]_q=1+q+\cdots+q^{m-1}$.  In the difference of formula~\cite[(2.1)]{PahujaCorrelations}, with
		indices $(a,b)$, and $q$ times formula~\cite[(2.2)]{PahujaCorrelations}, with indices $(b,a)$, the
		paired correction terms simplify by
		\begin{equation}\label{eq:q-integer-pairing}
			([m]_q-1)+q^m=q[m]_q.
		\end{equation}
		For $d>1$, multiplication by $D$ therefore gives
		\begin{align}\label{eq:current-simplification}
			D\bigl(\widehat c_{a,b}^q(n)-q\widehat c_{b,a}^q(n)\bigr)
			&=2d-q(n-1)+q(A-B-C),
		\end{align}
		where
		\begin{align*}
			A&=(d+1)\bigl(2b(n-a)+a+b-n-1\bigr),\\
			B&=(d+2)(b-1)(n-a),\\
			C&=bd(n-a+1).
		\end{align*}
		A direct expansion, using $a=b+d$, yields
		\begin{equation}\label{eq:pahuja-polynomial-cancellation}
			A-B-C=n-1-2d.
		\end{equation}
		Consequently, \eqref{eq:current-simplification} equals $2d(1-q)$.  When $d=1$, the
		last term in formula~(2.2) is zero, as specified there, and the corresponding
		last term in formula~(2.1) also vanishes because $[1]_q-1=0$.  In this case,
		\cite[Theorem~2.1]{PahujaCorrelations} and the same pairing give
		\begin{align*}
			D\bigl(\widehat c_{b+1,b}^q(n)-q\widehat c_{b,b+1}^q(n)\bigr)
			&=2-q\bigl((n-1)+b(n-b)\bigr)+q(A-B)\\
			&=2(1-q),
		\end{align*}
		where now $A-B=n-3+b(n-b)$.  Thus, in all cases,
		\begin{equation}\label{eq:pahuja-current-combination}
			\widehat c_{a,b}^q(n)-q\widehat c_{b,a}^q(n)
			=\frac{2(1-q)(a-b)}{n^2(n-1)}
			=\frac{(1-q)(a-b)}{n\binom n2}.
		\end{equation}
		Taking $a=n+1-i$ and $b=n+1-j$ in
		\eqref{eq:correlation-label-reversal}--
		\eqref{eq:pahuja-current-combination} proves
		\eqref{eq:ring-current-identity}.
	\end{proof}
	
	\section{Separation of leaders and the limit shape}\label{sec:limit-shape}
	
	In this section, we prove our main results. The main new technical obstacle for us was the proof of the fact that the affine random walk, after a finite random time, can be coupled with multi-species ASEP on a ring\footnote{geometrically, this can be interpreted as the fact that the affine random walk eventually stays in a Weyl chamber forever. In this paper, we chose to work in a permutation language, though all statements can also be formulated geometrically.}. For $q=0$ this is obvious, but in our case, such a stabilization is less obvious and required a detailed proof. We prove it in Proposition \ref{prop:eventual-separation} below, after some preparatory work.  
	
	\subsection{Tagged-particle displacements on the ring}
	\label{subsec:ring-displacements}
	
	Let $(\omega_t)_{t\ge0}$ be the multispecies ASEP on the ring.  We write
	$P_t(\omega,\cdot)$ for its transition kernel.
	
	\begin{lemma}\label{lem:ring-mixing}
		For every $\varepsilon>0$, there exists $T_\varepsilon<\infty$ such that
		\begin{equation}\label{eq:ring-mixing}
			\sup_{\omega\in S_n}
			\norm{P_t(\omega,\cdot)-\pi_n^q}_{\TV}<\varepsilon,
			\qquad t\ge T_\varepsilon.
		\end{equation}
	\end{lemma}
	
	\begin{proof}
		This is the standard convergence theorem for an irreducible continuous-time
		Markov chain on a finite state space.  Uniformity over the initial state is
		automatic because $S_n$ is finite; see, for example,
		\cite[Chapter~4]{LevinPeres}.
	\end{proof}
	
	For $i\in[n]$, let $X_i(t)$ be the lifted displacement of ring particle $i$:
	every clockwise jump adds $1$ and every counterclockwise jump subtracts $1$ from it.
	The initial lifted position is immaterial, and we set $X_i(0)=0$.  Let
	$g_i(\omega)$ be the total rate at which particle $i$ jumps clockwise minus
	the total rate at which it jumps counterclockwise in state $\omega$.  Then
	\begin{equation}\label{eq:ring-martingale}
		M_i(t)=X_i(t)-\int_0^t g_i(\omega_s)\,ds
	\end{equation}
	is a square-integrable martingale.  Since a tagged particle is affected only
	by its two adjacent clocks, its predictable quadratic variation is bounded
	by $2t$.
	
	\begin{lemma}\label{lem:ring-velocity}
		For every initial ring state and every $i\in[n]$,
		\begin{equation}\label{eq:ring-velocity-lln}
			\frac{X_i(t)}{t}\longrightarrow v_i
			\qquad\text{almost surely and in }L^1,
		\end{equation}
		where
		\begin{equation}\label{eq:ring-velocities}
			v_i=\frac{(1-q)(n+1-2i)}{n-1}.
		\end{equation}
		In particular,
		\begin{equation}\label{eq:strict-velocity-order}
			v_1>v_2>\cdots>v_n.
		\end{equation}
	\end{lemma}
	
	\begin{proof}
		Since $\langle M_i\rangle_t\le2t$, the martingale strong law gives
		$M_i(t)/t\to0$ almost surely.  Moreover,
		$\Ee[M_i(t)^2]=\Ee[\langle M_i\rangle_t]\le2t$, so
		$M_i(t)/t\to0$ in $L^2$, and hence in $L^1$.  By the ergodic theorem for the
		finite irreducible chain,
		\[
		\frac1t\int_0^t g_i(\omega_s)\,ds
		\longrightarrow \pi_n^q(g_i)
		\]
		almost surely.  Since $g_i$ is bounded, the same convergence holds in
		$L^1$.  It remains to compute the stationary drift.
		
		At stationarity, the clockwise and counterclockwise jump rates of particle
		$i$ are, respectively,
		\begin{align*}
			r_i&=n\left(\sum_{\substack{j\in[n]\\j>i}}c_{i,j}^q(n)
			+q\sum_{\substack{j\in[n]\\j<i}}c_{i,j}^q(n)\right),\\
			\ell_i&=n\left(\sum_{\substack{j\in[n]\\j<i}}c_{j,i}^q(n)
			+q\sum_{\substack{j\in[n]\\j>i}}c_{j,i}^q(n)\right).
		\end{align*}
		Indeed, the events
		$\{\omega(r)=a,\omega(r+1)=b\}$, $r\in\Z/n\Z$, are disjoint when viewed as
		events specifying the position of particle $a$.  Rotation invariance therefore
		gives
		\[
		\Pp_{\pi_n^q}(\text{the clockwise neighbor of }a\text{ is }b)
		=\sum_{r\in\Z/n\Z}\Pp_{\pi_n^q}(\omega(r)=a,\omega(r+1)=b)
		=n c_{a,b}^q(n).
		\]
		These rate factors then follow directly from
		\eqref{eq:ring-rates}. One obtains,
		\begin{align*}
			\pi_n^q(g_i)
			&=n\sum_{\substack{j\in[n]\\j>i}}
			\bigl(c_{i,j}^q(n)-q c_{j,i}^q(n)\bigr)
			-n\sum_{\substack{j\in[n]\\j<i}}
			\bigl(c_{j,i}^q(n)-q c_{i,j}^q(n)\bigr)\\
			&=\frac{1-q}{\binom n2}
			\left(\sum_{j=i+1}^n(j-i)-\sum_{j=1}^{i-1}(i-j)\right)\\
			&=\frac{1-q}{\binom n2}
			\left(\binom{n+1-i}{2}-\binom{i}{2}\right)
			=\frac{(1-q)(n+1-2i)}{n-1},
		\end{align*}
		where the second equality is Theorem~\ref{thm:ring-current-identity}.
		This proves \eqref{eq:ring-velocity-lln}; the strict ordering follows from
		$v_i-v_{i+1}=2(1-q)/(n-1)>0$.
	\end{proof}
	
	\subsection{Leaders and the safe-region coupling}\label{subsec:leaders}
	
	Let $(w_t)_{t\ge0}$ be the affine Hecke walk.  For $i\in[n]$, the particle
	with label $i$ is called the $i$th \emph{leader}, and its position at time $t$ is
	\begin{equation}\label{eq:leader-position}
		Y_i(t)=w_t^{-1}(i),
		\qquad i\in[n].
	\end{equation}
	The $n$ leader positions determine the whole affine permutation, since
	\begin{equation}\label{eq:leader-class-periodicity}
		w_t^{-1}(i+kn)=Y_i(t)+kn,
		\qquad i\in[n],\quad k\in\Z.
	\end{equation}
	
	Define the \textit{safe} region
	\begin{equation}\label{eq:safe-region}
		\cD=\left\{w\in\wtS:
		\min_{\substack{i,j\in[n]\\i\ne j}}
		\abs{w^{-1}(i)-w^{-1}(j)}>n\right\}.
	\end{equation}
	Suppose $w\in\cD$, write $Y_i=w^{-1}(i)$, and order the leaders from right
	to left:
	\begin{equation}\label{eq:right-to-left-order}
		Y_{a_1}>Y_{a_2}>\cdots>Y_{a_n},
		\qquad (a_1,\ldots,a_n)\in S_n.
	\end{equation}
	Note that the residues of $Y_1,\ldots,Y_n$ are pairwise distinct modulo $n$.  Define a ring word
	$\Omega(w)\in S_n$ by
	\begin{equation}\label{eq:safe-ring-word}
		\Omega(w)(r)=k
		\quad\Longleftrightarrow\quad
		Y_{a_k}\equiv r\pmod n,
		\qquad r,k\in[n].
	\end{equation}
	Thus ring label $k$ records the $k$th leader in right-to-left order.
	
	\begin{lemma}[Safe-region coupling]\label{lem:safe-region-coupling}
		Start the affine walk from $w\in\cD$ and a ring ASEP from $\Omega(w)$, using
		the same marked clocks.  Let $X_k(t)$ be the lifted displacement of ring
		particle $k$, with $X_k(0)=0$.  Up to the first time at which the affine walk
		leaves $\cD$,
		\begin{equation}\label{eq:safe-coupling}
			Y_{a_k}(t)=Y_{a_k}(0)+X_k(t),
			\qquad k\in[n].
		\end{equation}
		Moreover, if the $n$ paths on the right-hand side of
		\eqref{eq:safe-coupling} remain pairwise at distance greater than $n$ for all
		$t\ge0$, then the affine walk never leaves $\cD$ and
		\eqref{eq:safe-coupling} holds for all $t\ge0$.
	\end{lemma}
	
	\begin{proof}
		Set $b_i:=Y_i-i$ for $i\in[n]$.  Every representative of leader class $i$
		has the form $x=Y_i+mn$ and carries the integer label $i+mn=x-b_i$.
		If $Y_i>Y_j+n$, then
		\begin{equation}\label{eq:b-order}
			b_i-b_j=Y_i-Y_j-(i-j)>1,
		\end{equation}
		because $i-j\le n-1$.  Hence, if adjacent sites $x,x+1$ carry
		representatives of classes $i$ and $j$, respectively, then
		\begin{equation}\label{eq:adjacent-label-comparison}
			(x-b_i)-(x+1-b_j)=-1-(b_i-b_j)<0.
		\end{equation}
		Thus, whenever leader $i$ is farther to the right than leader $j$, a
		representative of class $i$ immediately to the left of a representative of
		class $j$ has the smaller integer label.  Reversing the two classes reverses
		the comparison.  Therefore, as long as the leaders remain in $\cD$, the
		exchange rate of two adjacent leader classes is $1$ when their right-to-left
		ranks are increasing and $q$ when those ranks are decreasing, exactly as in
		\eqref{eq:ring-rates}.
		Thus, the two systems make identical accepted or rejected moves under the common mark, which concludes the proof. 
	\end{proof}
	
	The next proposition is the key to our proof of Proposition~\ref{prop:eventual-separation}: it shows that if leaders are far apart at some point, then with positive probability, they will be far apart forever. 
	
	\begin{proposition}\label{prop:leaders-separate}
		For every $\varepsilon>0$, there exists $M=M(\varepsilon,n,q)$ such that,
		from every initial state satisfying
		\begin{equation}\label{eq:M-separated-initial}
			\min_{\substack{i,j\in[n]\\i\ne j}}
			\abs{Y_i(0)-Y_j(0)}\ge M,
		\end{equation}
		one has
		\begin{equation}\label{eq:survival-safe-probability}
			\Pp\bigl(w_t\in\cD\text{ for every }t\ge0\bigr)
			\ge1-\varepsilon.
		\end{equation}
	\end{proposition}
	
	\begin{proof}
		Order the leaders in the initial state according to \eqref{eq:right-to-left-order}, and consider
		the coupled ring process started from $\Omega(w_0)$.  For a ring initial
		state $\omega\in S_n$, define
		\begin{equation}\label{eq:ring-backtracking-variable}
			R_\omega=
			\max_{1\le k<\ell\le n}
			\sup_{t\ge0}\bigl(X_\ell(t)-X_k(t)\bigr).
		\end{equation}
		By Lemma~\ref{lem:ring-velocity}, for $k<\ell$,
		\[
		\frac{X_\ell(t)-X_k(t)}{t}\longrightarrow v_\ell-v_k<0.
		\]
		The trajectories have finitely many jumps on compact time intervals.  Since
		the displayed ratio has a strictly negative limit, the difference is
		negative and of linear order for all sufficiently large times; its supremum
		is therefore attained on a finite time interval.  Hence
		$R_\omega<\infty$ almost surely.  Since $S_n$ is finite, there exists a
		deterministic integer $R<\infty$ such that
		\begin{equation}\label{eq:uniform-backtracking-quantile}
			\inf_{\omega\in S_n}\Pp_\omega(R_\omega\le R)
			\ge1-\varepsilon.
		\end{equation}
		Choose $M=n+1+R$ and set $E=\{R_{\Omega(w_0)}\le R\}$. By
		\eqref{eq:uniform-backtracking-quantile}, $\Pp(E)\ge1-\varepsilon$.
		On $E$, for every $k<\ell$ and every $t\ge0$,
		\[
		\bigl(Y_{a_k}(0)+X_k(t)\bigr)
		-\bigl(Y_{a_\ell}(0)+X_\ell(t)\bigr)
		\ge M-R=n+1>n.
		\]
		Therefore Lemma~\ref{lem:safe-region-coupling} implies that the affine walk
		remains in $\cD$ forever on $E$. This proves
		\eqref{eq:survival-safe-probability}. The bound is uniform in the initial
		affine state because \eqref{eq:uniform-backtracking-quantile} is uniform
		over all possible initial ring words.
	\end{proof}
	
	\subsection{Uniform reachability of separated configurations}
	\label{subsec:uniform-reachability}
	
	For $M\ge1$, set
	\begin{equation}\label{eq:SM-definition}
		\cS_M=\left\{w\in\wtS:
		\min_{\substack{i,j\in[n]\\i\ne j}}
		\abs{w^{-1}(i)-w^{-1}(j)}\ge M\right\}.
	\end{equation}
	
	The following claim is very natural and is close to being standard. However, we have not found an exact reference, and therefore provide a full proof. 
	
	\begin{proposition}\label{prop:reach-separated}
		For every $M\ge1$, set
		\begin{equation}\label{eq:CM-explicit}
			C_M=\ceil{\frac{M}{2n}}\frac{n(n^2-1)}{3}.
		\end{equation}
		For every $w\in\wtS$, there exist residues
		$r_1,\ldots,r_{C_M}\in\Z/n\Z$ such that
		\begin{equation}\label{eq:increasing-path}
			\ell(ws_{r_1}\cdots s_{r_k})=\ell(w)+k,
			\qquad 1\le k\le C_M,
		\end{equation}
		and
		\begin{equation}\label{eq:path-reaches-SM}
			ws_{r_1}\cdots s_{r_{C_M}}\in\cS_M.
		\end{equation}
		Therefore, for the affine Hecke walk started from $w$, one has
		\begin{equation}\label{eq:uniform-reach-probability}
			\inf_{w\in\wtS}\Pp_w(w_{C_M}\in\cS_M)
			\ge e^{-nC_M}>0.
		\end{equation}
	\end{proposition}
	
	\begin{proof}
		Put $u=w^{-1}=[y_1,\ldots,y_n]$ (recall the window notation from Section~\ref{subsec:affine-hecke}). Note $y_i=Y_i$ in the notation of Section~\ref{subsec:leaders}. Let $\rho(i)$ be the rank of $y_i$
		among the window entries, with rank $1$ assigned to the smallest entry.  Write
		\[
		K=\ceil{\frac{M}{2n}},
		\qquad c_r=2r-n-1\quad(r\in[n]),
		\]
		and, with residue zero represented by $n$, define
		\[
		d_{\modn{y_i}}=Kc_{\rho(i)},\qquad i\in[n].
		\]
		The residues of the $y_i$'s are pairwise distinct, and $\sum_{r=1}^n c_r=0$.
		Thus $d=(d_1,\ldots,d_n)\in\Z^n$ is well defined, $\sum_r d_r=0$, and
		\[
		\theta_d(r+kn)=r+n(k+d_r),
		\qquad r\in[n],\quad k\in\Z,
		\]
		defines an affine permutation.  If $y_i<y_j$, then
		\[
		(\theta_du)(j)-(\theta_du)(i)
		=y_j-y_i+2nK\bigl(\rho(j)-\rho(i)\bigr)
		\ge 1+2nK\ge M.
		\]
		Hence the window entries of $\theta_du$ are pairwise $M$-separated.
		
		It remains to check that $\theta_du$ lies above $u$ in left weak order.  For
		$1\le i<j\le n$, set
		\[
		A_{ij}=\floor{\frac{y_j-y_i}{n}},
		\qquad
		D_{ij}=d_{\modn{y_j}}-d_{\modn{y_i}}.
		\]
		By construction, $D_{ij}$ has the same sign as $y_j-y_i$; consequently,
		$A_{ij}$ and $D_{ij}$ have the same weak sign, and
		$|A_{ij}+D_{ij}|-|A_{ij}|=|D_{ij}|$.  Applying the window formula
		\eqref{eq:affine-length-window} to $u$, $\theta_du$, and $\theta_d$, and
		using that the residues of the $y_i$'s form a permutation of $[n]$, gives
		\begin{align*}
			\ell(\theta_du)-\ell(u)
			&=\sum_{1\le i<j\le n}|D_{ij}|\\
			&=\sum_{1\le r<s\le n}|d_s-d_r|
			=\ell(\theta_d).
		\end{align*}
		Moreover, since the multiset of the $d_r$'s is
		$\{Kc_1,\ldots,Kc_n\}$,
		\[
		\ell(\theta_d)
		=2K\sum_{1\le r<s\le n}(s-r)
		=K\frac{n(n^2-1)}{3}=C_M.
		\]
		
		Let $\tau=\theta_d^{-1}$ and choose a reduced expression
		$\tau=s_{r_1}\cdots s_{r_{C_M}}$.  Taking inverses in
		$\ell(\theta_du)=\ell(\theta_d)+\ell(u)$ yields
		$\ell(w\tau)=\ell(w)+\ell(\tau)$.  Hence every prefix of the chosen reduced
		word is length-additive, proving \eqref{eq:increasing-path}; and
		$(w\tau)^{-1}=\theta_du$, so the separation proved above gives
		\eqref{eq:path-reaches-SM}.
		
		Finally, on each interval $(k-1,k]$, require the clock $r_k$ to ring exactly
		once and all other clocks not to ring.  This event has probability $e^{-n}$
		per interval, hence $e^{-nC_M}$ altogether.  Every prescribed move is
		length-increasing and is therefore accepted independently of its mark.  This
		proves \eqref{eq:uniform-reach-probability}.
	\end{proof}
	
	\begin{proposition}\label{prop:eventual-separation}
		From every initial affine permutation, and in particular from the identity,
		there exists almost surely a finite random time $T_\mathrm{sep}$ such that
		\begin{equation}\label{eq:eventual-safe-region}
			w_t\in\cD,
			\qquad \mbox{for all $\ t \ge T_\mathrm{sep}$ }.
		\end{equation}
	\end{proposition}
	
	\begin{proof}
		Let
		\[
		A=\left\{\text{there exists }T<\infty\text{ such that }
		w_t\in\cD\text{ for all }t\ge T\right\}.
		\]
		Apply Proposition~\ref{prop:leaders-separate} with $\varepsilon=1/2$, and
		let $M$ be the resulting separation scale.  Proposition
		\ref{prop:reach-separated} gives constants $C_M<\infty$ and
		$\delta=e^{-nC_M}>0$ such that
		\[
		\Pp_w(w_{C_M}\in\cS_M)\ge\delta
		\qquad\text{for every }w\in\wtS.
		\]
		By the Markov property at time $C_M$ and Proposition
		\ref{prop:leaders-separate},
		\begin{equation}\label{eq:uniform-eventual-success}
			h(w):=\Pp_w(A)\ge\frac{\delta}{2}=:p>0,
			\qquad w\in\wtS.
		\end{equation}
		
		Let $(\cF_t)$ be the natural filtration of the affine Hecke walk.  Removing a finite initial segment
		of a path does not change the event $A$.  Therefore the time-homogeneous
		Markov property gives, for every integer $m\ge0$,
		\begin{equation}\label{eq:conditional-tail-probability}
			\Pp(A\mid\cF_m)=h(w_m)\ge p
			\qquad\text{almost surely}.
		\end{equation}
		The chain is nonexplosive because its total clock rate is $n$, and hence
		$A\in\cF_\infty:=\sigma(\bigcup_{m\ge0}\cF_m)$.  By L\'evy's upward theorem,
		the left-hand side of \eqref{eq:conditional-tail-probability} converges
		almost surely to $\one_A$.  Passing to the limit in the lower bound gives
		$\one_A\ge p$ almost surely.  Since $p>0$, the event $A^c$ has probability
		zero, which is precisely the assertion that a finite
		$T_\mathrm{sep}$ satisfying \eqref{eq:eventual-safe-region} exists.
	\end{proof}
	
	\subsection{Limit shape and local equilibrium}\label{subsec:limit-shape}
	
	For $K\in\{0,1,\ldots,n\}$, let $\nu_{n,K}$ be the probability law on
	$\{0,1\}^{\Z}$ obtained by choosing a uniformly random $K$-element subset
	$B\subset\Z/n\Z$ and setting
	\begin{equation}\label{eq:nu-nK}
		\eta(x)=\one_{\{\modn{x}\in B\}},
		\qquad x\in\Z.
	\end{equation}
	Thus $\nu_{n,K}$ is the uniform stationary law of the $K$-particle ASEP on
	the $n$-site ring, lifted periodically to $\Z$.
	
	\begin{theorem}[Limit shape and local equilibrium for step initial data]
		\label{thm:limit-shape-step}
		Let the affine Hecke walk start from $w_0=\id$, and let
		\begin{equation}\label{eq:chi-from-wt}
			\chi_t(x)=\one_{\{w_t(x)\le0\}}
		\end{equation}
		be its one-species projection.  Recall the leader velocities
		\begin{equation}\label{eq:theorem-velocities}
			v_k=\frac{(1-q)(n+1-2k)}{n-1},
			\qquad k\in[n].
		\end{equation}
		Then the following assertions hold.
		
		\smallskip
		\noindent\emph{(i) Leader law of large numbers.}
		Almost surely, there is a random permutation $(a_1,\ldots,a_n)$ such that,
		for all sufficiently large $t$,
		\begin{equation}\label{eq:eventual-leader-order}
			Y_{a_1}(t)>Y_{a_2}(t)>\cdots>Y_{a_n}(t),
		\end{equation}
		and
		\begin{equation}\label{eq:leader-velocity-limit}
			\frac{Y_{a_k}(t)}{t}\longrightarrow v_k,
			\qquad k\in[n].
		\end{equation}
		
		\smallskip
		\noindent\emph{(ii) Integrated limit shape.}
		Let
		\begin{equation}\label{eq:particle-count}
			N_t(x)=\sum_{y>x}\chi_t(y),
			\qquad x\in\Z.
		\end{equation}
		Then, almost surely,
		\begin{equation}\label{eq:uniform-height-limit}
			\sup_{u\in\R}
			\abs{\frac{N_t(\floor{ut})}{t}-\cH(u)}
			\longrightarrow0,
		\end{equation}
		where
		\begin{equation}\label{eq:limit-height-function}
			\cH(u)=\frac1n\sum_{k=1}^n\pospart{v_k-u}.
		\end{equation}
		In particular, away from the break points $v_1,\ldots,v_n$, the macroscopic
		density $\rho(u)=-\cH'(u)$ is
		\begin{equation}\label{eq:density-profile}
			\rho(u)=
			\begin{cases}
				1,&u<v_n,\\[1mm]
				\dfrac{K}{n},&v_{K+1}<u<v_K,
				\quad 1\le K\le n-1,\\[2mm]
				0,&u>v_1.
			\end{cases}
		\end{equation}
		
		\smallskip
		\noindent\emph{(iii) Local equilibrium.}
		Let $x_t\in\Z$ be deterministic and satisfy $x_t/t\to u$, where
		$u\notin\{v_1,\ldots,v_n\}$.  Let $K=K(u)$ be the unique integer in
		$\{0,\ldots,n\}$ such that $\rho(u)=K/n$.  Then, for every fixed
		$N\ge1$,
		\begin{equation}\label{eq:local-equilibrium}
			\bigl(\chi_t(x_t+j)\bigr)_{j=0}^{N-1}
			\ \Longrightarrow\ 
			\bigl(\eta(j)\bigr)_{j=0}^{N-1},
			\qquad \eta\sim\nu_{n,K}.
		\end{equation}
		Equivalently, for every bounded function $f:\{0,1\}^N\to\R$,
		\begin{equation}\label{eq:local-equilibrium-test-functions}
			\lim_{t\to\infty}
			\Ee\left[f\bigl((\chi_t(x_t+j))_{j=0}^{N-1}\bigr)\right]
			=\int f\bigl((\eta(j))_{j=0}^{N-1}\bigr)\,\nu_{n,K}(d\eta).
		\end{equation}
	\end{theorem}
	
	\begin{proof}
		We prove the three parts in turn.
		
		\medskip
		\noindent\emph{Step 1: asymptotic velocities of the leaders.}
		For each deterministic integer $T\ge0$, construct an auxiliary ring ASEP
		from the graphical marks after time $T$.  If $w_T\in\cD$, order the leaders
		from right to left, form the ring word $\Omega(w_T)$ as in
		\eqref{eq:safe-ring-word}, and start the ring process from this state.  If
		$w_T\notin\cD$, start it from an arbitrary fixed state.  Denote the lifted
		displacement of ring particle $k$ in this auxiliary process by
		$X_k^{(T)}(s)$.
		
		Conditionally on $\cF_T$, the initial ring state is fixed and the future
		marks are independent of $\cF_T$.  Lemma~\ref{lem:ring-velocity} therefore
		implies, with conditional probability one,
		\[
		\frac{X_k^{(T)}(s)}s\longrightarrow v_k,
		\qquad k\in[n].
		\]
		For each fixed $T$ this holds on an event of probability one; intersecting
		these events over $T\in\Z_{\ge0}$ preserves probability one.
		
		By Proposition~\ref{prop:eventual-separation}, almost surely there is an
		integer $T$ such that $w_t\in\cD$ for every $t\ge T$.  Fix such a $T$ on the
		full-probability event just constructed, and write $a_1,\ldots,a_n$ for the
		right-to-left order of the leaders at time $T$.  This order cannot change
		while the walk remains in $\cD$.  Lemma~\ref{lem:safe-region-coupling}
		gives, for every $s\ge0$,
		\begin{equation}\label{eq:theorem-safe-coupling}
			Y_{a_k}(T+s)-Y_{a_k}(T)=X_k^{(T)}(s),
			\qquad k\in[n].
		\end{equation}
		Dividing by $T+s$ and letting $s\to\infty$ proves
		\eqref{eq:leader-velocity-limit}, and hence part~(i).
		
		\medskip
		\noindent\emph{Step 2: the integrated limit shape.}
		For each $i\in[n]$, the particles whose labels are congruent to $i$ modulo
		$n$ have labels and positions
		\begin{equation}\label{eq:class-particles}
			i+nk
			\quad\text{and}\quad
			Y_i(t)+nk,
			\qquad k\in\Z,
		\end{equation}
		respectively.  Since $i+nk\le0$ if and only if $k\le-1$, the contribution of
		this residue class to $N_t(x)$ is
		\begin{equation}\label{eq:class-count-exact}
			\#\{k\le-1:Y_i(t)+nk>x\}
			=\pospart{\ceil{\frac{Y_i(t)-x}{n}}-1}.
		\end{equation}
		For every $z\in\R$,
		\[
		\abs{\pospart{\ceil{z/n}-1}-\pospart{z/n}}\le1.
		\]
		Summing over the $n$ residue classes therefore gives
		\begin{equation}\label{eq:height-approx-leaders}
			\abs{N_t(x)-\frac1n\sum_{i=1}^n\pospart{Y_i(t)-x}}
			\le n.
		\end{equation}
		Since $(a-u)_+$ is $1$-Lipschitz in each variable,
		\eqref{eq:height-approx-leaders}, part~(i), and
		$\abs{\floor{ut}/t-u}\le t^{-1}$ yield
		\begin{align*}
			&\sup_{u\in\R}
			\abs{\frac{N_t(\floor{ut})}{t}
				-\frac1n\sum_{k=1}^n\pospart{v_k-u}}\\
			&\qquad\le \frac nt+\frac1t
			+\max_{1\le k\le n}
			\abs{\frac{Y_{a_k}(t)}t-v_k}
			\longrightarrow0.
		\end{align*}
		This proves \eqref{eq:uniform-height-limit}.  Formula
		\eqref{eq:density-profile} follows by differentiating
		\eqref{eq:limit-height-function} away from its break points.
		
		\medskip
		\noindent\emph{Step 3: local equilibrium.}
		For a deterministic integer $T\ge0$, set
		\begin{equation}\label{eq:AT-definition}
			A_T=\{w_s\in\cD\text{ for every }s\ge T\}.
		\end{equation}
		The events $A_T$ increase with $T$, and Proposition
		\ref{prop:eventual-separation} implies
		\begin{equation}\label{eq:AT-probability}
			\Pp(A_T)\longrightarrow1
			\qquad(T\to\infty).
		\end{equation}
		Let $\omega^{(T)}$ be the auxiliary ring process from Step~1 and define its
		$K$-particle projection by
		\begin{equation}\label{eq:ring-K-projection}
			\zeta_s^{(T,K)}(r)
			=\one_{\{\omega_s^{(T)}(r)\le K\}},
			\qquad r\in\Z/n\Z.
		\end{equation}
		
		On $A_T$, let $a_1,\ldots,a_n$ be the right-to-left order of the leaders at
		time $T$.  The safe-region coupling implies
		\begin{equation}\label{eq:ring-label-at-leader-residue}
			\omega_{t-T}^{(T)}\bigl(\modn{Y_{a_k}(t)}\bigr)=k,
			\qquad k\in[n],\quad t\ge T.
		\end{equation}
		Indeed, ring particle $k$ starts at residue
		$\modn{Y_{a_k}(T)}$ and its lifted displacement is
		$Y_{a_k}(t)-Y_{a_k}(T)$.
		
		Fix $y\in\Z$, and let $k$ be the unique index for which
		$y\equiv Y_{a_k}(t)\pmod n$.  For a unique $m\in\Z$,
		\[
		y=Y_{a_k}(t)+nm,
		\qquad
		w_t(y)=a_k+nm.
		\]
		As $a_k\in[n]$, it follows that
		\begin{equation}\label{eq:occupation-via-leader}
			\chi_t(y)=1
			\quad\Longleftrightarrow\quad
			m\le-1
			\quad\Longleftrightarrow\quad
			Y_{a_k}(t)\ge y+n.
		\end{equation}
		Use the conventions $v_0=+\infty$ and $v_{n+1}=-\infty$.  By the definition
		of $K$,
		\[
		v_K>u>v_{K+1}.
		\]
		Part~(i), $x_t/t\to u$, and the fact that $N$ is fixed imply that, almost
		surely on $A_T$, for all sufficiently large $t$, simultaneously for
		$0\le j<N$,
		\begin{equation}\label{eq:block-coupling}
			\chi_t(x_t+j)
			=\zeta_{t-T}^{(T,K)}\bigl(\modn{x_t+j}\bigr).
		\end{equation}
		
		It remains to pass from this pathwise comparison to the unconditional
		limiting law. Set
		\[
		d(s)=\sup_{\omega\in S_n}
		\norm{P_s(\omega,\cdot)-\pi_n^q}_{\TV}.
		\]
		By Lemma~\ref{lem:ring-mixing}, $d(s)\to0$ as $s\to\infty$.
		The projection $\omega\mapsto(\one_{\{\omega(r)\le K\}})_{r\in\Z/n\Z}$
		sends $\pi_n^q$ to the uniform law on $K$-particle ring configurations.
		
		For fixed $T$, define the two blocks
		\[
		B_t=\bigl(\chi_t(x_t+j)\bigr)_{j=0}^{N-1},\qquad
		B_t^{(T)}=
		\bigl(\zeta_{t-T}^{(T,K)}(\modn{x_t+j})\bigr)_{j=0}^{N-1},
		\qquad t\ge T.
		\]
		The initial state of $\omega^{(T)}$ is $\cF_T$-measurable, and its driving
		marks after time $T$ are independent of $\cF_T$. Thus its unconditional
		law at time $t-T$ is within $d(t-T)$ of $\pi_n^q$ in total variation.
		For a bounded function $f:\{0,1\}^N\to\R$, write
		\[
		c_f=\int f\bigl((\eta(j))_{j=0}^{N-1}\bigr)\,\nu_{n,K}(d\eta).
		\]
		Rotation invariance of the uniform ring law handles the possibly varying
		residue of $x_t$, and gives
		\[
		\left|\Ee[f(B_t^{(T)})]-c_f\right|
		\le 2\norm{f}_\infty d(t-T).
		\]
		On $A_T$, equation~\eqref{eq:block-coupling} holds for all sufficiently
		large $t$, almost surely. Dominated convergence therefore implies
		\[
		\limsup_{t\to\infty}\Pp(B_t\ne B_t^{(T)})\le\Pp(A_T^c).
		\]
		Combining these estimates, for every fixed deterministic integer $T$,
		\[
		\limsup_{t\to\infty}\left|\Ee[f(B_t)]-c_f\right|
		\le2\norm{f}_\infty\Pp(A_T^c).
		\]
		Finally let $T\to\infty$ and use \eqref{eq:AT-probability}. This proves
		\eqref{eq:local-equilibrium-test-functions}, and hence part~(iii).
	\end{proof}
	
	\begin{corollary}[One- and two-site local statistics]
		\label{cor:two-site-local-equilibrium}
		Under the assumptions of Theorem~\ref{thm:limit-shape-step}(iii),
		\begin{equation}\label{eq:one-point-limit}
			\Pp\bigl(\chi_t(x_t)=1\bigr)\longrightarrow\frac Kn=\rho(u).
		\end{equation}
		Moreover,
		\begin{equation}\label{eq:two-site-local-equilibrium}
			\bigl(\chi_t(x_t),\chi_t(x_t+1)\bigr)
			\ \Longrightarrow\ (\eta(0),\eta(1)),
			\qquad \eta\sim\nu_{n,K},
		\end{equation}
		and the limiting probabilities are
		\begin{align}
			\Pp_{\nu_{n,K}}(11)
			&=\frac{K(K-1)}{n(n-1)},
			\label{eq:adjacent-two-point-limit}\\
			\Pp_{\nu_{n,K}}(10)
			=\Pp_{\nu_{n,K}}(01)
			&=\frac{K(n-K)}{n(n-1)},
			\label{eq:adjacent-mixed-two-point-limit}\\
			\Pp_{\nu_{n,K}}(00)
			&=\frac{(n-K)(n-K-1)}{n(n-1)}.
			\label{eq:adjacent-empty-two-point-limit}
		\end{align}
	\end{corollary}
	
	\begin{proof}
		Apply Theorem~\ref{thm:limit-shape-step}(iii) with $N=1$ and $N=2$.
		The two adjacent sites represent distinct residues because $n\ge2$.
		The displayed probabilities follow by choosing a uniform $K$-element subset
		of the $n$ residues.
	\end{proof}
	
	Theorem \ref{thm:limit-shape-step} implies that the internal order at positions from 1 to N converges to the stationary ring ASEP measure. Therefore, the following fact directly follows from the color-position symmetry.
	\begin{corollary}[Distribution of the limiting leader order]
		\label{rem:leader-order-distribution}
		Let the affine Hecke walk start from the identity. For $w\in\wtS$, let $L(w)\in S_n$ record the left-to-right rank\footnote{ $L(w)$ coincides with the ranking function $\rho$ used in the proof of Proposition~\ref{prop:reach-separated}.} of each
		leader:
		\begin{equation}\label{eq:leader-rank-permutation}
			L(w)(i)=1+\#\{j\in[n]:w^{-1}(j)<w^{-1}(i)\},
			\qquad i\in[n].
		\end{equation}
		By Proposition~\ref{prop:eventual-separation}, $L(w_t)$ is eventually
		constant almost surely; denote its terminal value by $L_\infty$. One has
		\begin{equation}\label{eq:leader-order-stationary-law}
			\Law(L_\infty)=\pi_n^q,
		\end{equation}
	\end{corollary}
	
	\subsection{Second class particles}\label{subsec:second-class-particles}
	
	We use the ordered alphabet $\{1<2<+\infty\}$, where $1$ denotes a first
	class particle, $2$ a second class particle, and $+\infty$ a hole, and the
	projection
	\[
	\varphi(j)=
	\begin{cases}
		1,&j\le-1,\\
		2,&j=0,\\
		+\infty,&j\ge1.
	\end{cases}
	\]
	Under this projection, the affine Hecke walks started from $\id$ and from
	$s_0$ give the two initial conditions considered below.
	
	We shall also use the following elementary diagonal translation invariance.
	For $m\in\Z$, define
	\[
	(\Theta_m w)(x)=w(x+m)-m,
	\qquad x\in\Z.
	\]
	This is conjugation by the translation $x\mapsto x+m$; in particular,
	$\Theta_m$ maps $\wtS$ to itself and fixes the identity.  It conjugates the
	update at bond class $r$ into the update at bond class $r-m$.  Since the
	marked clocks of all bond classes have the same law,
	\begin{equation}\label{eq:diagonal-shift-invariance}
		\Theta_m w_t\ \stackrel d=\ w_t,
		\qquad m\in\Z,
	\end{equation}
	for the identity-started affine Hecke walk.
	
	\begin{theorem}[Second class particle at the step]
		\label{thm:second-class-step}
		Consider the periodic two-species ASEP with initial condition
		\begin{equation}\label{eq:second-class-step-initial}
			\eta_0(x)=
			\begin{cases}
				1,&x\le-1,\\
				2,&x=0,\\
				+\infty,&x\ge1.
			\end{cases}
		\end{equation}
		Let $Q(t)$ be the position of its unique second class particle.  For every
		$t\ge0$ and $x\in\Z$,
		\begin{equation}\label{eq:second-class-exact-identity}
			\Pp\bigl(Q(t)\le x\bigr)
			=\Pp\bigl(\chi_t(-x)=1\bigr),
		\end{equation}
		where $(\chi_t)$ is the one-species process of
		Theorem~\ref{thm:limit-shape-step}.  Consequently, if $x_t/t\to u$ and
		$u\notin\{-v_1,\ldots,-v_n\}$, then
		\begin{equation}\label{eq:second-class-step-limit}
			\Pp\bigl(Q(t)\le x_t\bigr)\longrightarrow\rho(-u).
		\end{equation}
		Equivalently,
		\begin{equation}\label{eq:second-class-uniform-speed}
			\frac{Q(t)}t\ \Longrightarrow\ V,
			\qquad
			\Pp(V=v_k)=\frac1n,
			\quad k\in[n].
		\end{equation}
	\end{theorem}
	
	\begin{proof}
		Let $w_t$ be the identity-started affine Hecke walk.  Under the three-color
		projection, the second class particle is the particle of color $0$, so
		$Q(t)=w_t^{-1}(0)$.  Color-position symmetry and
		\eqref{eq:diagonal-shift-invariance} give
		\begin{align*}
			\Pp(Q(t)\le x)
			&=\Pp(w_t^{-1}(0)\le x)
			=\Pp(w_t(0)\le x)\\
			&=\Pp\bigl((\Theta_xw_t)(-x)\le0\bigr)
			=\Pp(w_t(-x)\le0),
		\end{align*}
		which is \eqref{eq:second-class-exact-identity}.
		
		Now apply Theorem~\ref{thm:limit-shape-step}(iii) with $N=1$ at the
		macroscopic point $-u$ to obtain \eqref{eq:second-class-step-limit}.  Away from its jump points, the function $u\mapsto\rho(-u)$ agrees with
		the distribution function of the measure placing mass $1/n$ at each of the
		points $-v_k$, $k\in[n]$.  Since $v_{n+1-k}=-v_k$, this is the uniform
		measure on $\{v_1,\ldots,v_n\}$, proving
		\eqref{eq:second-class-uniform-speed}.
	\end{proof}
	
	\begin{theorem}[Second class particle in a one-site perturbation]
		\label{thm:shifted-second-class}
		Consider the periodic two-species ASEP with initial condition
		\begin{equation}\label{eq:shifted-second-class-initial}
			\widetilde\eta_0(x)=
			\begin{cases}
				1,&x\le-1,\\
				+\infty,&x=0,\\
				2,&x=1,\\
				+\infty,&x\ge2.
			\end{cases}
		\end{equation}
		Let $\widetilde Q(t)$ be the position of its unique second class particle.
		For every $t\ge0$ and $x\in\Z$,
		\begin{align}
			\Pp\bigl(\widetilde Q(t)\le x\bigr)
			&=\Pp\bigl(\chi_t(-x)=1,\ \chi_t(1-x)=1\bigr)\notag\\
			&\quad+q\,\Pp\bigl(\chi_t(-x)=0,\ \chi_t(1-x)=1\bigr)
			\label{eq:shifted-second-class-exact-identity}\\
			&=q\,\Pp\bigl(\chi_t(1-x)=1\bigr)
			+(1-q)\Pp\bigl(\chi_t(-x)=\chi_t(1-x)=1\bigr).
			\notag
		\end{align}
		Suppose that $x_t/t\to u$ and
		$u\notin\{-v_1,\ldots,-v_n\}$.  Let $K=K(-u)$ be characterized by
		$\rho(-u)=K/n$.  Then
		\begin{equation}\label{eq:shifted-second-class-limit}
			\Pp\bigl(\widetilde Q(t)\le x_t\bigr)
			\longrightarrow
			F_{n,q}(u)
			:=q\frac Kn+(1-q)\frac{K(K-1)}{n(n-1)}.
		\end{equation}
		Equivalently, $\widetilde Q(t)/t\Rightarrow\widetilde V$, where
		$\widetilde V$ is supported on $\{v_1,\ldots,v_n\}$ and
		\begin{equation}\label{eq:shifted-second-class-atoms-direct}
			\Pp(\widetilde V=v_k)
			=\frac qn+\frac{2(1-q)(n-k)}{n(n-1)},
			\qquad k\in[n].
		\end{equation}
	\end{theorem}
	
	\begin{proof}
		Let $\widetilde w_t$ be the affine Hecke walk started from $s_0$, and let
		\[
		H_t=\exp\left(t\sum_{r\in\Z/n\Z}(T_{s_r}-T_{\id})\right)
		\]
		be the Hecke-algebra transition element of the identity-started walk.  The
		law of $\widetilde w_t$ is encoded by $T_{s_0}H_t$.  Since
		$\iota(H_t)=H_t$ and $s_0^{-1}=s_0$, the anti-involution gives
		\begin{equation}\label{eq:shifted-inversion-algebra}
			\iota(T_{s_0}H_t)=H_tT_{s_0}.
		\end{equation}
		Consequently, $\widetilde w_t^{-1}$ has the law of the identity-started
		walk $w_t$ followed by one independent stochastic Hecke update at bond
		class $0$.
		
		The second class particle has color $0$, and hence
		$\widetilde Q(t)=\widetilde w_t^{-1}(0)$.  Before the final update, put
		\[
		\xi_j=\one_{\{w_t(j)\le x\}},
		\qquad j=0,1.
		\]
		If $(\xi_0,\xi_1)=11$, position $0$ is occupied after the update; if the
		pair is $10$ or $00$, it is not occupied.  For the pair $01$, the underlying
		colors form a descent, and position $0$ becomes occupied exactly when the
		exchange is accepted, which has probability $q$.  Therefore
		\[
		\Pp(\widetilde Q(t)\le x)
		=\Pp(\xi_0=\xi_1=1)+q\Pp(\xi_0=0,\xi_1=1).
		\]
		By diagonal translation invariance,
		\[
		(\xi_0,\xi_1)
		\ \stackrel d=\ 
		(\chi_t(-x),\chi_t(1-x)),
		\]
		which proves \eqref{eq:shifted-second-class-exact-identity}.  This is the
		periodic analogue of the finite-perturbation color--position argument in
		\cite[Section~5]{BorodinBufetovColorPosition}.
		
		Apply Corollary~\ref{cor:two-site-local-equilibrium} at the macroscopic
		point $-u$.  The one-site and $11$ limits give
		\[
		\lim_{t\to\infty}\Pp(\widetilde Q(t)\le x_t)
		=q\frac Kn+(1-q)\frac{K(K-1)}{n(n-1)},
		\]
		which is \eqref{eq:shifted-second-class-limit}.
		
		Write the right-hand side as $F_K$, $0\le K\le n$.  As $u$ increases
		through $-v_K$, the integer $K(-u)$ increases from $K-1$ to $K$, and
		\[
		F_K-F_{K-1}
		=\frac qn+\frac{2(1-q)(K-1)}{n(n-1)}.
		\]
		These differences are nonnegative, $F_0=0$, and $F_n=1$, so they are the
		atoms of the limiting distribution.  This proves
		\eqref{eq:shifted-second-class-atoms-direct}.
	\end{proof}
	
	% Later parts of Section 3 will be added in a subsequent version.

\end{document}